\documentclass[preprint]{elsarticle}
\usepackage{amssymb,amsfonts,amsmath,amsthm}
\usepackage{float}
\usepackage{enumerate}
\usepackage{mathrsfs}
\usepackage{graphicx}
\usepackage{hyperref}
\usepackage{tikz}
\usepackage[shortlabels]{enumitem}
\usetikzlibrary{cd}
\usepackage{caption}
\usepackage{breqn}
\usepackage{subcaption}

\newtheorem{thm}{Theorem}[section]

\newtheorem{conj}[thm]{Conjecture}
\newtheorem{quest}[thm]{Question}

\theoremstyle{definition}
\newtheorem{defn}[thm]{Definition}

\theoremstyle{remark}
\newtheorem{rk}[thm]{Remark}

\renewcommand{\P}{\mathbb{P}}

\newcommand{\C}{\mathbb{C}}

\newcommand{\N}{\mathbb{N}}
\newcommand{\PP}{\mathcal{P}}
\newcommand{\EE}{\mathcal{E}}
\newcommand{\bmx}{\begin{matrix}}
\newcommand{\emx}{\end{matrix}}
\newcommand{\bbmx}{\begin{bmatrix}}
\newcommand{\ebmx}{\end{bmatrix}}
\newcommand{\bpmx}{\begin{pmatrix}}
\newcommand{\epmx}{\end{pmatrix}}
\newcommand{\bvmx}{\begin{vmatrix}}
\newcommand{\evmx}{\end{vmatrix}}
\newcommand{\ub}{\underbrace}

\newcommand{\wt}{\widetilde}
\newcommand{\ul}{\underline}

\newcommand{\inc}{\subset}

\newcommand{\setm}{\setminus}
\newcommand{\tto}{\rightarrow}

\newcommand{\ol}{\overline}
\newcommand{\ot}{\otimes}

\newcommand{\OT}{\bigotimes}

\DeclareMathOperator{\Con}{Con}

\DeclareMathOperator{\TNS}{TNS}
\DeclareMathOperator{\inn}{in}
\DeclareMathOperator{\out}{out}
\DeclareMathOperator{\QMF}{QMF}

\newcommand{\Flat}{\mathrm{Flat}}
\newcommand{\mc}{\mathrm{mc}}

\newcommand{\ital}{\textit}

\title{On the geometry of Tensor Network States of $2 \times N$ Grids}
\author[add1]{Parth Sarin}
\ead{psarin@stanford.edu}
\address[add1]{Student in Department of Mathematics at Stanford University}

\begin{document}
\begin{abstract}

We discuss the geometry of a class of tensor network states, called projected entangled pair states in the Physics literature. We provide initial results towards a question of Verstraete and Rizzi regarding the tensor network state of an $M \times N$ grid; we partially answer the question for a $2 \times N$ grid. We also study the $2 \times N$ grid sitting on a torus and provide initial results towards understanding the Zariski closure of the set of tensor network states associated to this graph. Finally, we give explicit tensors that provide optimal (using currently available methods) bounds on the border rank of a generic tensor in the tensor network state of the $2 \times N$ grid.

\end{abstract}


\maketitle
\tableofcontents

\section{Introduction} A \textit{tensor network} associated to a graph is a recipe for constructing tensors in large spaces from tensors in small spaces. Let $G = (V,E)$ be a graph with dangling edges (i.e., edges that are only connected to one vertex). For example, if $G$ is regular of degree $n$, with one dangling edge at each vertex, a tensor network state allows us to construct tensors in $(\C^d)^{\ot |V|}$ from tensors in $(\C^d)^{\ot n}$. Typically, $d$ and $n$ are small (e.g., $d=2$, $n=3$ or $4$) and $|V|$ is large.

For this reason, tensor networks can represent tensor information more ``efficiently'' since in practical implementations, the space $(\C^d)^{\ot n}$ will be much smaller than $(\C^d)^{\ot |V|}$ (see, e.g., \cite{MR3236394}). From the perspective of algebraic geometry, tensor networks provide a natural way of constructing varieties of tensors that are of interest and tractable. In solid state physics, tensor network states are used to model the interactions of particles at a quantum level. They are often used to approximate high-dimensional data like those which arise in the study of quantum entanglement (see, e.g., \cite{MR3244614}).

Ideally, one would like a complete geometric description of the tensors in $(\C^d)^{\ot |V|}$ attainable with a given graph $G$. It is also of interest to give a geometric description of the boundary tensors in $(\C^d)^{\ot |V|}$ to the tensors attainable with a graph $G$ or to give a description of the Zariski closure of the set of tensors attainable with a graph. However, using currently available tools, when a graph has nontrivial topology, all three of these goals are likely out of reach. For example, in the simplest case, if $G$ is a critical, spoked, 3-cycle, the Zariski closure of the set of all tensor network states is the closure of the set of degenerations of matrix multiplication, which is famously a mysterious set (especially in the sense of \cite{MR2865915} regarding secant varieties over the Segre). 

In this paper, we restrict our attention to the $2 \times N$ grid on a torus. For the above reasons, one should not expect a complete understanding of the tensor network states associated to this graph. In this paper, we partially answer the Verstraete--Rizzi question, posed by Frank Verstraete (Universit\"at Wien, Austria) and Matteo Rizzi (Johannes Gutenberg-Universit\"at Mainz, DE) at the July Workshop on Quantum Physics and Geometry in Levico Terme, Italy in 2017. In the physics sense, this setup can be thought of as a model of $2N$ particles arranged in a $2 \times N$ grid on a torus. The results we provide in this paper give insight into the behavior of the entanglement interactions these particles would have. In particular, our results show that the entanglement behavior is ``optimal'' in the sense of Verstraete and Rizzi. These issues are discussed in several papers including \cite{michalek2018tensor} where a more detailed explanation is given of their significance in physics. We also prove bounds on the border rank for a generic tensor in that tensor network state.

\subsection{Tensor network states} Given a graph $G$ as described above, we call the dangling edges \textit{physical edges} and the non-dangling edges \textit{entanglement edges}. Let $\PP, \EE$ be collections of edges consisting of the physical edges and entanglement edges respectively. If $W, \mathbf{V}$ are vector spaces, let $\Con : W \ot \mathbf{V} \ot \mathbf{V}^\ast \tto W$ denote the contraction map (defined linearly by $w \ot T \ot \Phi \mapsto \Phi(T) w$). We depict $\mathbf{V}$ as bolded to signify that it might have additional structure as a space of tensors.

\begin{defn} Let $f : E \tto \N$. Suppose $G$ has $l$ physical edges denoted $p_1, \dots, p_l$. Let $W_1, \dots, W_l$ be complex vector spaces associated to the physical edges with $\dim(W_i) = f(p_i)$. Assign orientations to the entanglement edges of $G$ by choosing an orientation for each edge (the specific choice will not matter for this construction). For a vertex $v$, let $\inn(v)$ denote the set of entanglement edges directed towards $v$ and $\out(v)$ denote the set of edges directed away. 

Suppose $v \in V$. Associate to $v$ the vector space
\[ \mathbf{V}_v := \OT_{\substack{p_i \in v \\ p_i \in \PP}} W_i \ot \OT_{\substack{e \in v \\ e \in \inn(v)}} \C^{f(e)} \ot \left( \OT_{\substack{e \in v \\ e \in \out(v)}} \C^{f(e)} \right)^\ast. \]

Associated to a collection of $T_v \in \mathbf{V}_v$, there is a tensor $t \in W_1 \ot \cdots \ot W_l$ defined by $t = \Con(\OT_{v \in V} T_v)$. The set of all tensor network states associated to $G$ weighted by $f$ is called the \ital{(unbound) tensor network state}, denoted $\TNS(G,f) \inc W_1 \ot \cdots \ot W_l$, and $t \in \TNS(G,f)$ if for every $v \in V$, there exists a tensor $T_v \in \mathbf{V}_v$ such that $t = \Con(\OT_{v \in V} T_v)$.
\end{defn}

Practically, we often choose
\[ f(e) = \begin{cases}
	d & \text{ if } e \text{ is an entanglement edge} \\
	k & \text{ if } e \text{ is a physical edge}
\end{cases}. \]
For this setup, we denote $\TNS(G,f)$ by $\TNS(G,d,k)$. Unfortunately, this differs from the usual notation in the physics literature. Here the roles of ``$d$'' and ``$k$'' are reversed from how they are typically used in physics.

\begin{defn}
	Let $G$ be a regular graph. Define $f : E \tto \N$ by
\[ f(e) = \begin{cases}
	d & \text{ if } e \text{ is an entanglement edge} \\
	k & \text{ if } e \text{ is a physical edge}
\end{cases} \]
	where $d$ and $k$ are fixed. Again, define the spaces $\mathbf{V}_v$ as above for each vertex in the graph $G$. The constraints on $G$ and $f$ force $\mathbf{V}_v \simeq \mathbf{V}_w$ for any two vertices $v,w$ in $G$.
	
	Then, associated to a tensor $T \in \mathbf{V}_v$, there is a tensor $t \in W_1 \ot \cdots \ot W_l$ defined by $t = \Con(\OT_{v \in V} T)$. The set of all tensor network states associated to a regular graph $G$ with such a weight function $f$ is called the \ital{bound tensor network state}. It's denoted $\TNS^\prime (G, d, k) \inc W_1 \ot \cdots \ot W_l$ and $t \in \TNS^\prime (G, d, k)$ if there is a (single) tensor $T \in \mathbf{V}_v$ such that $t = \Con(\OT_{v \in V} T)$. The word ``bound'' refers to the restriction that all tensors must be the same.
\end{defn}

There is a natural \textit{construction function} associated to each of the tensor network states:
\begin{align*}
	\psi &: \bigoplus_{v \in V} \mathbf{V}_v \tto \OT_{p_i \in \PP} W_i \\
	\phi &: \mathbf{V}_v \tto \OT_{p_i \in \PP} W_i
\end{align*}
$\psi$ constructs $\TNS(G,f)$ and $\phi$ constructs $\TNS^\prime (G,f)$.

\subsection{Grid graphs and the Verstraete--Rizzi question} There has been recent interest in the following two graphs:

\begin{figure}[H]
\centering
\begin{subfigure}{.45\textwidth}
  \centering
  \includegraphics[width=5cm]{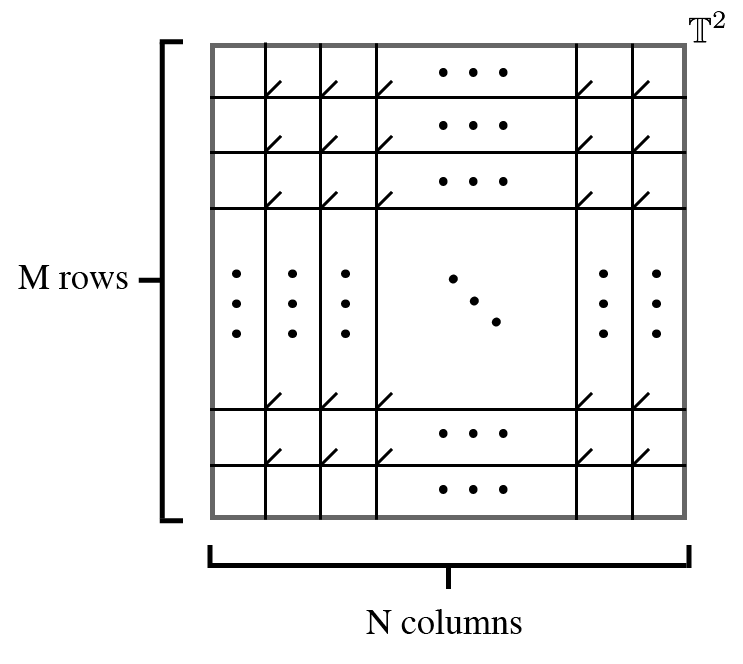}
  \caption{An $M \times N$ grid on a torus, denoted as $G_{M \times N}$, that has $MN$ physical edges.}
  \label{fig:N,M-grid-torus}
\end{subfigure}%
\hspace{.05\textwidth}
\begin{subfigure}{.45\textwidth}
  \centering
  \includegraphics[width=5cm]{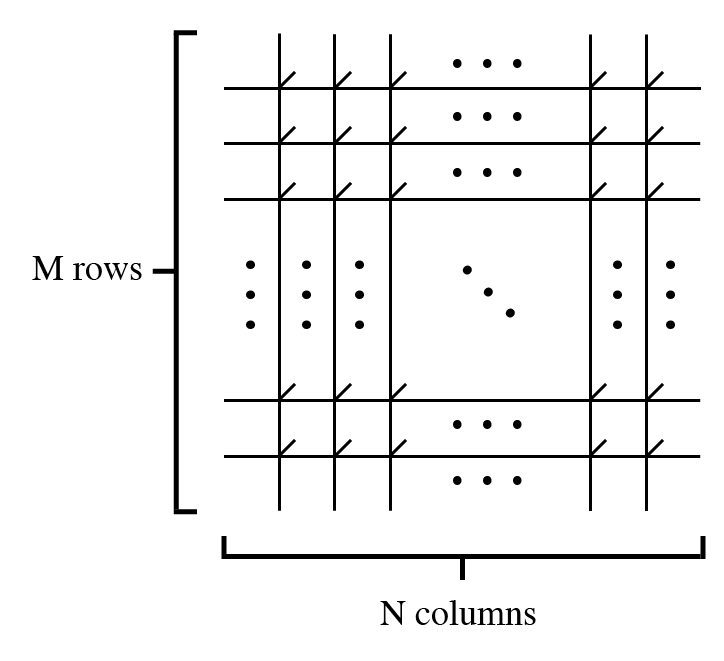}
  \caption{An $M \times N$ grid that is not on a torus, denoted $\wt{G}_{M \times N}$, that has $MN + 2M + 2N$ physical edges.}
  \label{fig:N,M-grid-dangling-edges}
\end{subfigure}
\end{figure}

In $\wt{G}_{2 \times N}$, even though the outer edges are physical edges, we often want them to have different weights than the other physical edges. This is because one should think of the external dangling edges as coming from ``outside sources'' (e.g., a machine in a lab that builds quantum states) and the internal dangling edges as the states one is trying to manipulate. We often choose
\[ f(e) = \begin{cases}
	d & \text{ if } e \text{ is an entanglement edge} \\
	k & \text{ if } e \text{ is a non-external physical edge} \\
	s & \text{ if } e \text{ is an external physical edge}
\end{cases}. \]
Denote this setup as $\TNS(\wt{G}_{2 \times N}, d, k, s) := \TNS(\wt{G}_{2 \times N}, f)$.

In particular, at the July Workshop, Frank Verstraete and Matteo Rizzi asked the following question:

\begin{quest} [The Verstraete--Rizzi question] If $M = N$, for what values of $d$ and $k$ does a generic tensor $T \in \TNS(\wt{G}_{N \times N}, d, k, d)$ have injective (i.e. maximal rank) edge/vertex flattening $T: (\C^d)^{\ot 4N} \tto (\C^k)^{\ot N^2}$? \end{quest}

We partially answer the Verstraete--Rizzi question for $M = 2$ and $N$ generic and study the grid graph more generally. With regards to the question, we obtain the following:
\begin{thm} \label{thm:V-R-question}
	Let $s \in \N$. Then, a generic tensor $T \in \TNS^\prime(\wt{G}_{2 \times N}, s^2, s^2, s^2)$ has edge/vertex flattening that is full rank.
\end{thm}
\begin{thm} \label{thm:V-R-question-unbound}
	Let $d \geq k$. Then, a generic tensor $T \in \TNS(\wt{G}_{2 \times N}, 1, k, d)$ has edge/vertex flattening that is full rank.
\end{thm}

\begin{rk}
	For the setups in Theorems \ref{thm:V-R-question} and \ref{thm:V-R-question-unbound}, it is impossible for a generic tensor to have an injective flattening. Instead, the best possible approximation to injectivity is surjectivity: showing that the flattening has full rank shows that it is ``as injective as possible.'' 
\end{rk}

The grid graph is particularly complex because it is unlike other graphs that have been used as tensor network states such as those used in Matrix Product States (in the sense of, e.g., \cite{MR3244614}) or in Hierarchical Tensor Representation (in the sense of \cite{MR3236394}). The present case differs because here, the building-block spaces, $\mathbf{V}_v$ are also complicated, being the tensor product of five vector spaces.

We study this problem for the case that $M = 2$ and $N$ is arbitrary. In that sense, this paper is the next step towards the problem after \cite{arXiv:1801.09106}, which studies the case that $M = 1$ and $N$ arbitrary. This is an important step, though, because when $M = 1$, $\mathbf{V}_v \simeq A \ot B \ot B^\ast$, a space that is better understood than the present case where $\mathbf{V}_v \simeq A \ot B \ot B \ot B^\ast \ot B^\ast$. And, for $M \geq 2$, one still has $\mathbf{V}_v \simeq A \ot B \ot B \ot B^\ast \ot B^\ast$ so techniques we use in this paper might admit generalization.

\subsection{The geometry of $\TNS(G_{2 \times N}, d, k)$}
Our main goal is to obtain geometric information about $\TNS(G_{2 \times N}, 2, 2)$ and $\TNS^\prime (G_{2 \times N}, 2, 2)$, such as the border rank of generic tensors in each set and explicit parameterizations of large subsets of each of the sets.

\begin{defn}
	A tensor $T \in W_1 \ot \cdots \ot W_l$ has \ital{rank one} if there exist vectors $w_j \in W_j$ such that $T = w_1 \ot \cdots \ot w_l$.
\end{defn}
\begin{defn}
	Let $T \in W_1 \ot \cdots \ot W_l$. The \ital{rank} of $T$ is the smallest $r$ such that $T$ can be written as the sum of $r$ rank one tensors. We denote the rank of $T$ as $R(T) = r$
\end{defn}
\begin{defn}
	The border rank of $T$ is the smallest $r$ such that $T$ can be written as a limit of tensors of rank $r$. In particular, $T$ cannot be written as a limit of tensors of rank $s$ for any $s < r$. We denote the border rank of a tensor $T$ by $\ul{R}(T) = r$.
\end{defn}

A simple parameter count implies that the unbound tensor network state $\TNS(G_{2 \times N}, 2, 2)$ has dimension at most $64N$. We construct a submanifold that has dimension on the order of $N$, which enables us to prove the following results:
\begin{thm} \label{thm:rank-of-TNS}
	If $N$ is even and $T \in \TNS(G_{2 \times N}, 2, 2)$ is generic, then $\ul{R}(T) \geq 2^N$.
\end{thm}

\begin{rk}
	The border rank bound we obtain in Theorem \ref{thm:rank-of-TNS} is the best bound possible using matrix flattenings. This is because the best flattening is the $N,N$ flattening, which is a map $(\C^2)^{\ot N} \tto (\C^2)^{\ot N}$ and thus can have rank at most $2^N$. Moreover for any $T \in (\C^k)^{\ot 2N}$, there are no known methods that give better bounds on border rank since $2N$ is even. In particular, Young-Koszul flattenings do not provide better bounds (a parameter count verifies this but for a more detailed explanation see, e.g., Chapter 2.4 of \cite{MR3729273}; Proposition 2.4.2.1 is especially informative).
	
	Additionally, we explain in the proof why $N$ must be even.
\end{rk}

\begin{thm} \label{thm:rank-of-bound-TNS}
	If $N$ is even and $T \in \TNS^\prime(G_{2 \times N}, 2, 2)$ is generic, then $\ul{R}(T) \geq 2^{N - 1}$.
\end{thm}

\begin{rk}
	For the same reason as above, the Young-Koszul flattening will not provide better bounds for the border rank of a generic tnesor.
\end{rk}

\section{The Verstraete--Rizzi Question} Number the vertices of the $2 \times N$ grid according to the following picture:

{\centering
\includegraphics[width=.9\textwidth]{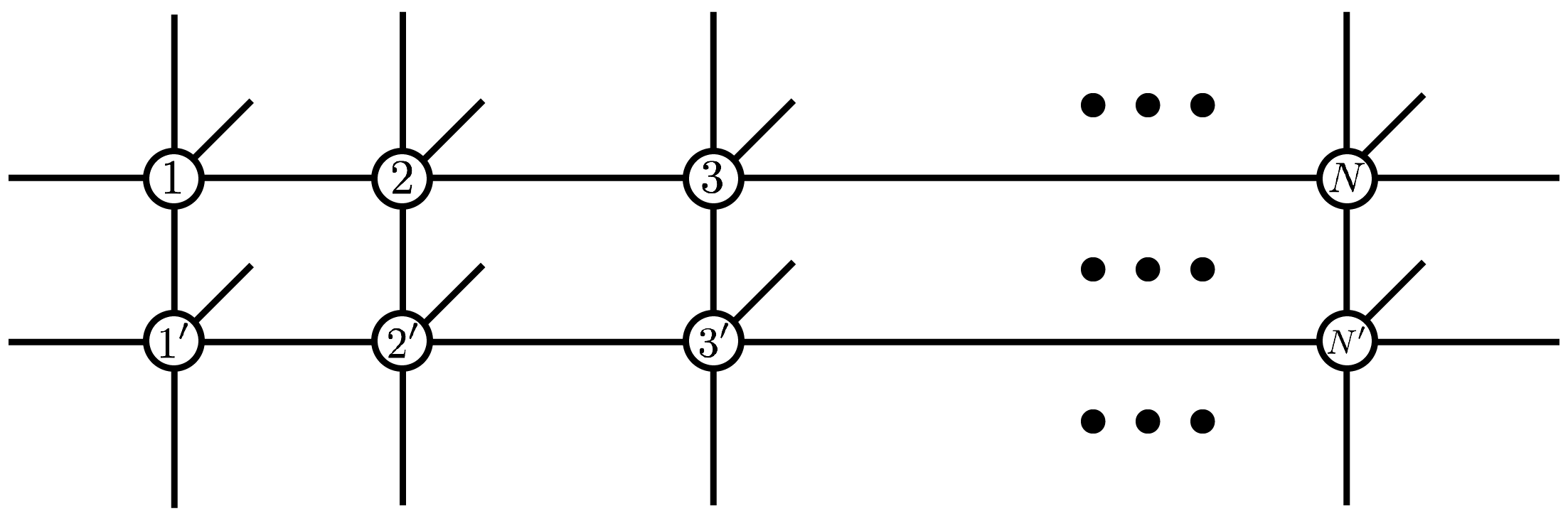}\par
}

We first prove Theorem \ref{thm:V-R-question}, which concerns the bound tensor network state.

\begin{proof} [Proof of Theorem \ref{thm:V-R-question}]
It is sufficient to show that there is a single tensor in $\TNS(G,s^2)$ with surjective edge/vertex flattening. Let $[s] = \{1, 2, \dots, s\}$. Place the iterated matrix multiplication tensor at each vertex indexed as:
\[T = \sum_{i,j,k,l,m \in [s]} a^i_j \ot b^m_i \ot c^k_l \ot d^j_k \ot e^l_m\]
Here, $a^i_j \in A$, $b^m_i \in B$, $c^k_l \in C$, $d^j_k \in D$, $e^l_m \in E$ are basis elements for their respective spaces and $A,B,C,D,E \simeq \C^{s^2} \simeq \C^s \ot \C^s$. Further, the vector space $A$ is attached to the upper edge, $B$ is attached to the physical edge, $C$ is attached to the right edge, $D$ is attached to the lower edge, and $E$ is attached to the left edge. We will denote the copy of $T$ placed at vertex $v$ by adding the subscript $v$ to each of the indices, i.e.
\[T_v = \sum_{i_v,j_v,k_v,l_v,m_v \in [s]} a^{i_v}_{j_v} \ot b^{m_v}_{i_v} \ot c^{k_v}_{l_v} \ot d^{j_v}_{k_v} \ot e^{l_v}_{m_v}\]

Let $\phi : (\C^{s^2})^{\ot 5} \tto (\C^{s^2})^{\ot 4N+4}$ be the construction function for $\TNS^\prime(\wt{G}_{2 \times N}, s^2, s^2, s^2)$. We claim that $\phi(T)$ has surjective edge/vertex flattening.

The edge/vertex flattening of $\phi(T)$ is a map
\[ \left(E^\ast \right)^{\ot 2} \ot \left(A^\ast \right)^{\ot N} \ot \left(C^\ast \right)^{\ot 2} \ot \left(D^\ast \right)^{\ot N} \tto B^{\ot 2N} \]
To show that it is surjective, we must show that every basis element of $B^{\ot 2N}$ is in the image of $\phi(T)$. Since contraction is a local operation, the proof takes place on a subgraph of $\wt{G}_{2 \times N}$:

{\centering
\includegraphics[width=.3\textwidth]{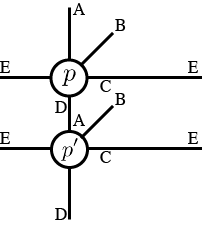}\par
}

The map $\phi$ performs two types of contractions: the left/right contractions and the up/down contractions. We use the symbol $\leftrightsquigarrow$ to denote the contraction. I.e., $V \leftrightsquigarrow W$ is the contraction of $V$ and $W$. So, we might write
\begin{align*}
	D_p \leftrightsquigarrow A_{p^\prime} &: A^{\ot N} \ot B^{\ot 2N} \ot C^{\ot 2} \ot D^{\ot N} \ot E^{\ot 2} \\
	&\tto A^{\ot (N-1)} \ot B^{\ot 2N} \ot C^{\ot 2} \ot D^{\ot (N-1)} \ot E^{\ot 2}
\end{align*} 
Since we've expressed $T$ in a basis, each contraction can be written as an equality between indices.

First consider the up/down contractions. For these, let $p = 1, 2, \dots, N$. The contractions give:
\[ D_p \leftrightsquigarrow A_{p^\prime}: \; \begin{cases}
	j_p = i_{p^\prime}  \\
	k_p = j_{p^\prime}
\end{cases} \]

Now consider the left-right contractions. These contractions give:
\[ C_p \leftrightsquigarrow E_{p+1}: \; \begin{cases}
	k_p = l_{p+1} \\
	l_p = m_{p+1}
\end{cases} \implies \begin{array}{l} 
	k_p = l_{p+1} = m_{p+2} \\ 
	l_1 = m_2
\end{array} \]
\[ C_{p^\prime} \leftrightsquigarrow E_{(p+1)^\prime}: \; \begin{cases}
	k_{p^\prime} = l_{(p+1)^\prime} \\
	l_{p^\prime} = m_{(p+1)^\prime}
\end{cases} \implies \begin{array}{l}
	k_{p^\prime} = l_{(p+1)^\prime} = m_{(p+2)^\prime} \\
	l_{1^\prime} = m_{2^\prime}
\end{array} \]

We claim that the indices on the $b$ vectors are all distinct. To verify this, adopt the ordering $1 < 1^\prime < 2 < 2^\prime < \cdots < N < N^\prime$ and replace all indices with the index of the \textit{lowest possible subscript}. For example the equation $k_p = l_{p+1} = m_{p+2}$ implies that $m_{p+2}$ should be replaced by $k_p$.

Two indices appear on the $b$ vectors. They are $i$ and $m$. First consider the $i$ index. The only contraction equation containing $i$ is the first: $j_p = i_{p^\prime}$. And, $j_p$ does not appear anywhere else in the equations. So, for all upper vertices, $i_p$ is left unchanged, and for all lower vertices, $i_{p^\prime}$ is replaced by $j_p$. Thusfar, all the indices on the $b$ vertices are still distinct.

Next, consider the $m$ index. For $1 \leq p \leq N-2$ we have:
\begin{align*}
	&k_p = m_{p+2}
	&k_{p^\prime} = m_{(p+2)^\prime}
\end{align*}

So, $m_1$ and $m_{1^\prime}$ are left untouched. Then, $m_2$ is replaced by $l_1$ and $m_{2^\prime}$ is replaced by $l_{1^\prime}$ (see the contractions $C \leftrightsquigarrow E$). And for all $r = 3, 3^\prime, 4, 4^\prime, \dots, N, N^\prime$, $m_r$ is replaced by $k_{r-2}$. Thus, all of the indices appearing on the $b$ vectors are distinct. 

We can ``control'' the value of each index in the image of the flattening by inputting a suitable choice of basis vectors in the source spaces. We can control the value of $i_p$ and $i_{p^\prime} = j_p$ by feeding the tensor the appropriate vector in $A_p^\ast$ since the indices on $a$ are $i$ and $j$. And, we can control the value of $m_1, m_{1^\prime}$ using $E_1^\ast, E_{1^\prime}^\ast$. We can control the value of $m_2 = l_1, m_{2^\prime} = l_{1^\prime}$ using $E_1^\ast, E_{1^\prime}^\ast$. Finally, for $p = 3, 4, \dots, N$, we can control the value of $m_p = k_{p-2} = j_{(p-1)^\prime}$ using $D_{(p-1)^\prime}^\ast$ and can control the value of $m_{p^\prime} = k_{(p-2)^\prime}$ using $D_{(p-2)^\prime}^\ast$. So, using various input vectors, we can obtain a basis for $B^{\ot 2N}$ in the image of the edge/vertex flattening of $\phi(T)$. Thus, the edge/vertex flattening is surjective, so it has full rank.
\end{proof}

Now we prove Theorem \ref{thm:V-R-question-unbound}, which concerns the unbound tensor network state.

\begin{proof} [Proof of Theorem \ref{thm:V-R-question-unbound}]
	Let $f = e_1$. Let $A_v, B_v, C_v, D_v, E_v$ be vector spaces indexed by the vertices of $\wt{G}_{2 \times N}$. Then, at a vertex of $\wt{G}_{2 \times N}$, let $A_v$ be associated to the upper edge, $B_v$ to the diagonal physical edge, $C_v$ to the right edge, $D_v$ to the lower edge, and $E_v$ to the left edge. The dimensions of $A_v ,C_v ,D_v ,E_v$ vary from vertex to vertex and $\dim(B_v) = k$ always. We must exhibit $2N$ tensors, $T_v \in A_v \ot B_v \ot C_v \ot D_v \ot E_v$, such that $\psi$ evaluated on the direct sum of those tensors has surjective edge/vertex flattening.
	
	Let
	\[ T = e_1 \ot e_1 \ot f \ot f \ot f + \cdots + e_k \ot e_k \ot f \ot f \ot f \]
	\[ T^\prime = f \ot e_1 \ot e_1 \ot f \ot f + \cdots + f \ot e_k \ot e_k \ot f \ot f \]
	Then, define $T_{v_1} = T_{v_2} = \cdots = T_{v_N} = T$. These are the tensors associated to the upper vertices of the graph. Define $T_{v_{1^\prime}} = T_{v_{2^\prime}} = \cdots = T_{v_{N^\prime}} = T^\prime$. These are the tensors associated to the lower vertices of the graph. Then, we can explicitly compute:
	\begin{align*}
		&\psi \left(T_{v_1} \oplus \cdots \oplus T_{v_N} \oplus T_{v_{1^\prime}} \oplus \cdots \oplus T_{v_{N^\prime}} \right) = \\
		&\sum_{\substack{i_1, \dots, i_N \in [k] \\ j_1, \dots, j_N \in [k]}} f^{\ot 4} \ot \ub{e_{i_1} \ot \cdots \ot e_{i_N}}_{\in A_1 \ot \cdots \ot A_N} \ot  \ub{e_{j_1} \ot \cdots \ot e_{j_N}}_{\in D_{1^\prime} \ot \cdots \ot D_{N^\prime}} \ot \ub{e_{i_1} \ot \cdots \ot e_{i_N} \ot  e_{j_1} \ot \cdots \ot e_{j_N}}_{\in B^{\ot 2N}}
	\end{align*}
	And thus, the flattening is surjective. In particular,
	\begin{align*} &\psi \left(T_{v_1} \oplus \cdots \oplus T_{v_N} \oplus T_{v_{1^\prime}} \oplus \cdots \oplus T_{v_{N^\prime}} \right)\left(f^\ast, f^\ast, f^\ast, f^\ast, e_{i_1}^\ast, \dots, e_{i_N}^\ast, e_{j_1}^\ast, \dots, e_{j_N}^\ast \right) = \\
	&\hspace{8mm}e_{i_1} \ot \cdots \ot e_{i_N} \ot  e_{j_1} \ot \cdots \ot e_{j_N}
	\end{align*}
	Thus, we can obtain a basis in the image of the flattening, so the image is the entire space.
\end{proof}

\section{The geometry of $\TNS(G_{2 \times N}, 2, 2)$} We want to exhibit large submanifolds of $\TNS(G_{2 \times N}, 2, 2)$. One strategy to do this is to pick tensors $T_v \in \mathbf{V}_v$ such that their projections to the entanglement edges are the same for every vertex, but the projection to the physical edge varies from vertex to vertex.

This strategy yields tensors in $\bigoplus_1^{2N} (\C^2)^{\ot 5}$ where we can understand the contraction behavior from the entanglement edges alone and then use the component along the physical edges to sweep out interesting spaces in $\TNS(G_{2 \times N}, 2, 2)$. 

For example suppose $N$ is even. Number the vertices of the $2 \times N$ grid in the same way as the proof of Theorem \ref{thm:V-R-question} and let $n \in \{1, 1^\prime, 2, 2^\prime, \cdots, N, N^\prime \}$. Let $A^n_i, B^n_i \in \C^2$ be general vectors for $i \in \{1, 2\}$. Then, consider the following tensors in $(\C^2)^{\ot 5}$, displayed pictorially where $(\C^2)^{\ot 5}$ is represented at a vertex of the $2 \times N$ grid:

{\centering
\includegraphics[width=.9\textwidth]{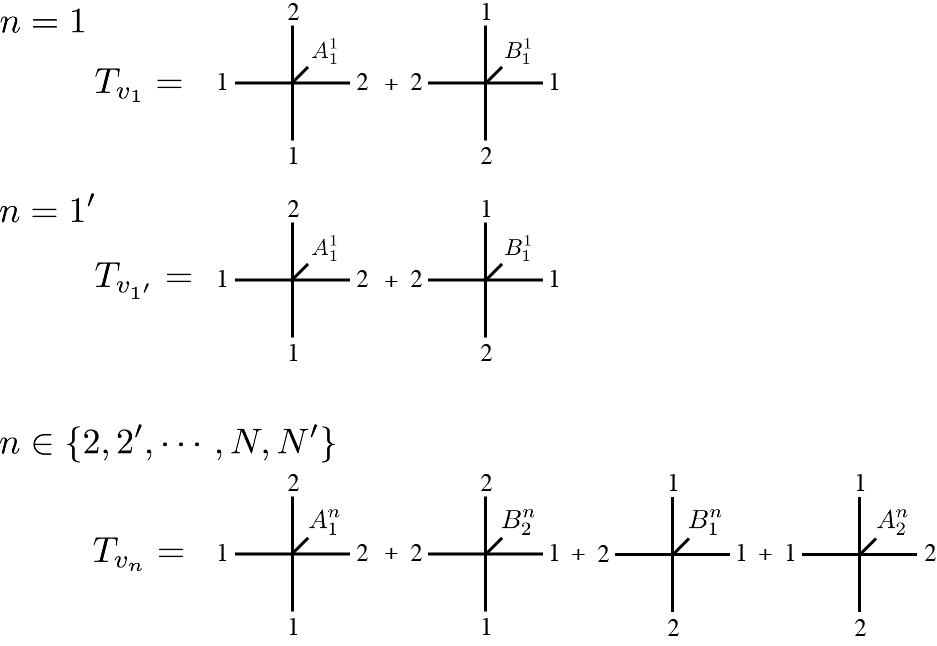}\par
}

Above, the numbers $1$ and $2$ represent basis vectors $e_1, e_2$ respectively. For example, 
\[ T_{v_1} = A_1^1 \ot e_1 \ot e_1 \ot e_2 \ot e_2 + B_1^1 \ot e_2 \ot e_2 \ot e_1 \ot e_1 \in (\C^2)^{\ot 5} \]

Consider $\psi(T_{v_1} \oplus T_{v_{1^\prime}} \oplus \cdots \oplus T_{v_{N^\prime}})$. We can represent it pictorially:

{\centering
\includegraphics[width=.9\textwidth]{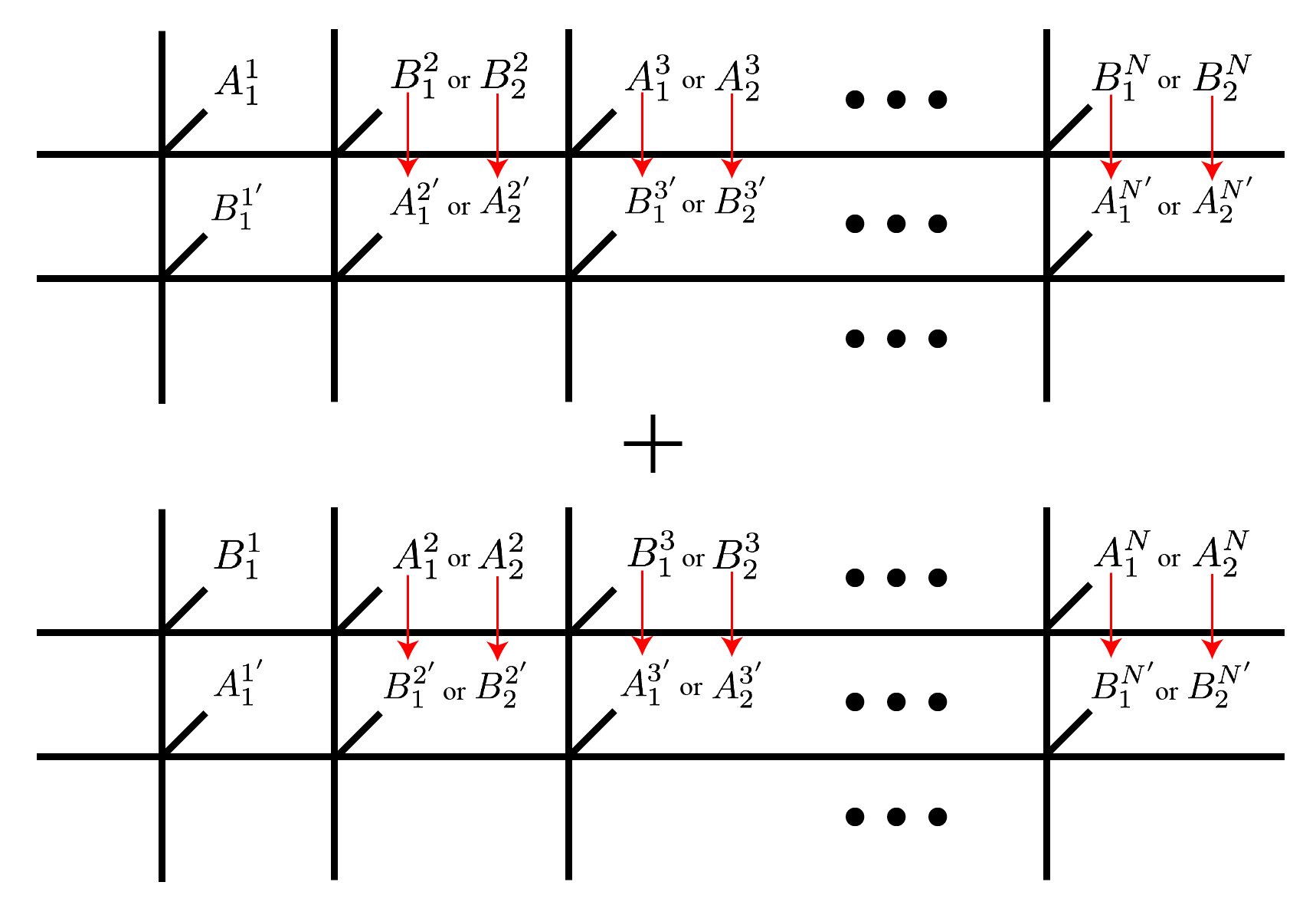}\par
}

Here, the red arrows indicate exclusive choices. In other words, define $F = \{ \gamma : \{2, 3, \dots, N \} \tto \{1,2\} \}$. Then:
\begin{align*}
	\psi(T_{v_1} \oplus T_{v_{1^\prime}} \oplus \cdots \oplus T_{v_{N^\prime}}) &= \sum_{\gamma \in F} \left( A^1_1 \ot B^{1^\prime}_1 \ot B^2_{\gamma(2)} \ot A^{2^\prime}_{\gamma(2)} \ot \cdots \ot B^N_{\gamma(N)} \ot A^{N^\prime}_{\gamma(N)} \right. \\
	&\hspace{8mm} \left. + B^1_1 \ot A^{1^\prime}_1 \ot A^2_{\gamma(2)} \ot B^{2^\prime}_{\gamma(2)} \ot \cdots \ot A^N_{\gamma(N)} \ot B^{N^\prime}_{\gamma(N)} \right)
\end{align*}

The picture shows why $N$ must be even. If $N$ were odd, then $B$ and $A$ would be switched at the far right, and contraction with the tensors at the far left would kill off the entire tensor.

This construction is key to the proof of Theorems \ref{thm:rank-of-TNS} and \ref{thm:rank-of-bound-TNS}.

\begin{proof} [Proof of Theorem \ref{thm:rank-of-TNS}] Let $N$ be even. Let $e_1, e_2$ be a basis of $\C^2$. Define $A^1_1 = e_1$, $B^1_1 = e_2$, $B^{1^\prime}_1 = e_1$, $A^{1^\prime}_1 = e_2$. For $n \in \{2, 2^\prime, \dots, N, N^\prime\}$, define $A^n_1 = B^{n}_1 = e_1$ and $A^n_2 = B^n_2 = e_2$. Let $T_{v_1}, T_{v_{1^\prime}}, T_{v_2}, T_{v_{2^\prime}}, \dots, T_{v_N}, T_{v_{N^\prime}}$ be as in the above construction. Let $T = T_{v_1} \oplus T_{v_{1^\prime}} \oplus \cdots \oplus T_{v_{N^\prime}}$.

Then,
\[ \psi(T) = \sum_{i_1, \dots, i_N \in \{1,2\}} e_{i_1} \ot e_{i_1} \ot e_{i_2} \ot e_{i_2} \ot \cdots \ot e_{i_N} \ot e_{i_N} \]

Consider the upper/lower flattening (from vertices $1, 2, \dots, N$ to $1^\prime, 2^\prime, \dots, N^\prime$). This flattening satisfies:
\[ \psi(T)(e_{i_1}, e_{i_2}, \dots, e_{i_N}) = e_{i_1} \ot e_{i_2} \ot \cdots \ot e_{i_N} \]

Thus, the flattening has rank $2^N$, so $\psi(T)$ has border rank at least $2^N$. We have an explicit expression for $\psi(T)$ as the sum of $2^N$ rank $1$ tensors, so in fact,
\[ 2^N = \ul{R}(\psi(T)) = R(\psi(T)) \]
Since border rank is a semi-continuous function, a generic tensor in $\TNS(G_{2 \times N}, 2, 2)$ has border rank at least $2^N$.
\end{proof}

To summarize, let $n \in \{1, 1^\prime, 2, 2^\prime, \cdots, N, N^\prime \}$, $i \in \{1, 2\}$, and pick $A^n_i, B^n_i \in \C^2$. If $F = \{ \gamma : \{2, 3, \dots, N\} \tto \{1,2\} \}$, then
\begin{equation}
\begin{split}
	&\sum_{\gamma \in F} \left( A^1_1 \ot B^{1^\prime}_1 \ot B^2_{\gamma(2)} \ot A^{2^\prime}_{\gamma(2)} \ot \cdots \ot B^N_{\gamma(N)} \ot A^{N^\prime}_{\gamma(N)} \right. \\
	&\hspace{8mm} \left. + B^1_1 \ot A^{1^\prime}_1 \ot A^2_{\gamma(2)} \ot B^{2^\prime}_{\gamma(2)} \ot \cdots \ot A^N_{\gamma(N)} \ot B^{N^\prime}_{\gamma(N)} \right) \in \TNS(G_{2 \times N}, 2, 2)
\end{split}
\end{equation}

Therefore, by a simple parameter count, the manifold consisting of all tensors of the above form is at most $8N-4$ dimensional. Another parameter count implies that $\TNS(G_{2 \times N}, 2, 2)$ is at most $2^5 \cdot 2N$ dimensional, since it is the image of $\psi$. 

We now turn our attention to the bound tensor network state and to Theorem \ref{thm:rank-of-bound-TNS}. 

\begin{proof} [Proof of Theorem \ref{thm:rank-of-bound-TNS}]
	Let $N$ be even. Let $A = C = e_1$ and let $B = D = e_2$. Define $T \in (\C^2)^{\ot 5}$ to be the following tensor:
	
	{\centering
	\includegraphics[width=.7\textwidth]{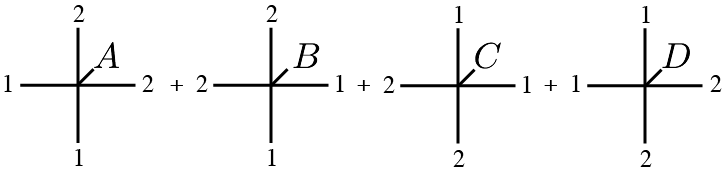} \par
	}
	
	Let $\phi : (\C^2)^{\ot 5} \tto (\C^2)^{\ot 2N}$ be the construction function for the bound tensor network state. Then consider the flattening of $\phi(T)$ depicted in the following graph coloring:
	
	{\centering
	\includegraphics[width=.7\textwidth]{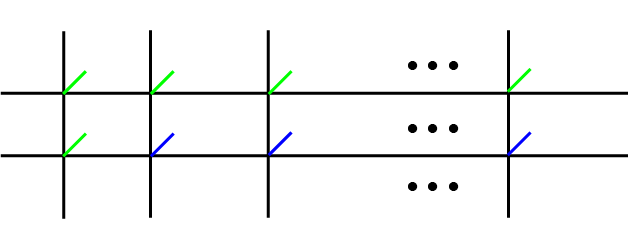} \par
	}
	
	The flattening is from the green edges to the blue edges. Recall the argument from the proof of Theorem \ref{thm:rank-of-TNS}. If we number the indices of the input of the flattening to correspond with the vertices in the following order: $1, 1^\prime, 2, 3, \dots, N$, then we have that
	\[ \phi(T)(e_1, e_1, e_{i_2}, e_{i_3}, \dots, e_{i_N}) = 2 e_{i_2} \ot e_{i_3} \ot \cdots \ot e_{i_N} \]
	And thus, the flattening rank of $\phi(T)$ in this $N+1,N-1$ flattening is $2^{N-1}$. The generic statement follows from the fact that border rank is sem-continuous.
\end{proof}

\section{Maximizing flattening rank} In the proof of Theorem \ref{thm:rank-of-TNS}, we found a tensor which gave optimal border rank bounds via its $N,N$ matrix flattening. It is still an open question to find an explicit tensor (explicit in the sense of complexity theory; see, e.g., \cite{MR2870721}) in $(\C^k)^{\ot 2N}$ where \textit{all} $N,N$ flattenings have maximal rank $k^N$.

There is no hope that such a tensor is in $\TNS(G_{2 \times N}, 2, 2)$ because the rank of a particular flattening for $T \in \TNS(G_{2 \times N}, 2, 2)$ is constrained by the max-flow/min-cut inequality. That is, if we consider the flattening from vertices $i_1, i_2, \dots, i_s \tto i_{s+1}, i_{s+2}, \dots, i_{2N}$, the rank of this flattening is less than or equal to
\[ 2^{\mc((i_1, i_2, \dots, i_s), (i_{s+1}, i_{s+2}, \dots, i_{2N}))} \]
Here, $\mc((i_1, i_2, \dots, i_s), (i_{s+1}, i_{s+2}, \dots, i_{2N}))$ denotes the size of the min-cut between the vertices $\{i_1, i_2, \dots, i_s\}$ and $\{i_{s+1}, i_{s+2}, \dots, i_{2N}\}$. For the general quantum max-flow/min-cut inequality see, e.g., Corollary 3.7 of \cite{MR3513725}.

So, for tensor network states, it is more interesting to study the following question:

\begin{quest} \label{quest:tns-flattening-qmf-question}
	For a graph $G$ with $2N$ physical edges and $d,k \in \N$, does there exist some $T \in \TNS(G, d, k)$ such that all flattenings of $T$ have the maximal rank possible, as constrained by the quantum max-flow/min-cut inequality? If not, is there some tensor whose $N,N$ flattenings have the maximal rank possible?
\end{quest}

This question can be phrased geometrically. Fix a graph $G$ and a function $f : G \tto E$. Let $V_1, V_2, \dots, V_n$ be the vector spaces associated to the physical edges of $G$. For $S \inc [n]$, define $\Flat_{S}^{\leq r}$ to be the projectivization of the set of tensors in $V_1 \ot V_2 \ot \cdots \ot V_n$ whose $S \tto [n] \setm S$ flattening has rank at most $r$. Define 
\[ \Flat_s^{\leq r} := \bigcap_{\substack{S \inc [n] \\ |S| = s}} \Flat_S^{\leq r}\] 
Finally, define $\QMF(G,f)$ to be the projectivization of the set of tensors in $V_1 \ot V_2 \ot \cdots V_n$ which obey the quantum max-flow/min-cut inequality. For some $s,r$, we have
\[ \ol{\P\TNS(G,f)} \inc \QMF(G,f) \inc \Flat_s^{\leq r} \]

In general $\TNS(G,f) \not= \QMF(G,f)$ because it is possible that $\TNS(G,f)$ does not contain any tensors that have maximal flattening rank for a given flattening (e.g., see \cite{arXiv:1801.09106}). Question \ref{quest:tns-flattening-qmf-question} asks when $\TNS(G, d, k) = \QMF(G, d, k)$.

Several empirical findings lead us to the following conjecture:

\begin{conj}
	For $d=2$ and $k$ sufficiently large, there exists a tensor $T \in \TNS(G_{2 \times N}, d, k)$ such that all flattenings of $T$ have maximal rank, as constrained by the quantum max-flow/min-cut inequality.
\end{conj}

We suspect this is true because we can explicitly construct tensors for which ``almost'' all flattenings have maximal rank and these tensors have different deficiencies (i.e., for different tensors, different flattenings drop rank). So, we suspect that one could increase $k$ and ``stitch'' together these tensors to obtain an ideal tensor.

One example of such a tensor is the following. Let $\phi : \C^4 \ot (\C^2)^{\ot 4} \tto (\C^4)^{\ot 2N}$ be the construction function for the bound tensor network state $\TNS^\prime(G_{2 \times N}, 2, 4)$. Let $A, B, C, D$ be a basis for $\C^4$. Then let $T$ be the following tensor in $\C^4 \ot (\C^2)^{\ot 4}$:

{\centering
\includegraphics[width=.7\textwidth]{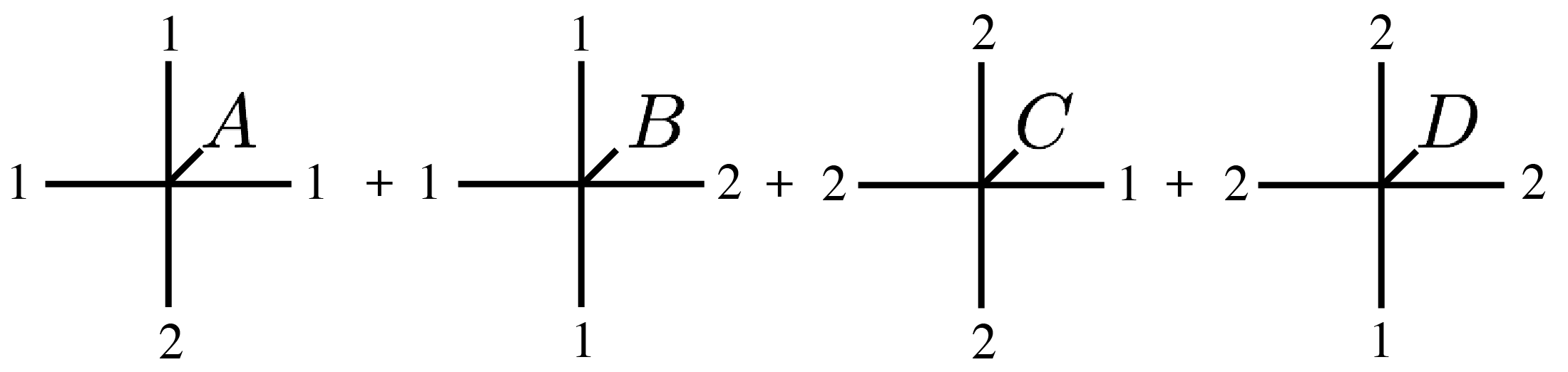}\par
}

This tensor has the property that if $A$ appears on a physical edge in the $2 \times N$ grid (i.e. in a term of the induced expression of $\phi(T)$), the contraction forces either $A$ or $B$ to appear to the right of that edge. It also forces either $A$ or $C$ to appear to the left. Either $B$ or $D$ can appear above, and either $C$ or $D$ can appear below. When we fully categorize which vectors can ``survive'' to the left/right/top/bottom of $A,B,C,D$, we obtain the following:

\begin{table}[h]
	\begin{tabular}[t]{c | c | c}
		& B,D & \\ \hline
		A,C & \textbf{A} & A,B \\ \hline
		& C,D & \\
	\end{tabular}
	\hfill
	\begin{tabular}[t]{c | c | c}
		 & B,D & \\ \hline
		A,C & \textbf{B} & C,D \\ \hline
		& B,A & \\
	\end{tabular}
	\hfill
	\begin{tabular}[t]{c | c | c}
		 & A,C & \\ \hline
		B,D & \textbf{C} & A,B \\	 \hline
		& C,D & \\
	\end{tabular}
\end{table}
\vspace{-5mm}
\begin{center}
	\begin{tabular}[t]{c | c | c}
		 & A,C & \\ \hline
		B,D & \textbf{D} & C,D \\ \hline
		& B,A & \\
	\end{tabular}
\end{center}

In particular, since the grid is $2 \times N$, and ``above'' and ``below'' are the same, this simplifies:
\begin{table}[h]
	\begin{tabular}[t]{c | c | c}
		& D & \\ \hline
		A,C & \textbf{A} & A,B
	\end{tabular}
	\hfill
	\begin{tabular}[t]{c | c | c}
		 & B & \\ \hline
		A,C & \textbf{B} & C,D
	\end{tabular}
	\hfill
	\begin{tabular}[t]{c | c | c}
		 & C & \\ \hline
		B,D & \textbf{C} & A,B
	\end{tabular}
	\hfill
	\begin{tabular}[t]{c | c | c}
		 & A & \\ \hline
		B,D & \textbf{D} & C,D
	\end{tabular}
\end{table}

One can see that locally, this tensor allows us to ``transmit'' the maximal number of linearly independent vectors over each entanglement edge ($2$ to the left, right, top, and bottom). Globally, we observe a cancelling deficiency that causes some flattenings to drop rank. Unfortunately, this indicates that the local behavior of the contraction is not enough to make a statement about the global behavior of the flattenings. That is, to answer Question \ref{quest:tns-flattening-qmf-question} for $G_{2 \times N}$, one must deal with the entire space $\bigoplus_{n=1}^{2N} (\C^d)^{\ot 4} \ot \C^k$. So, given that this space is very large, we suspect that  this question is out of reach (at least for $G_{2 \times N}$) using existing techniques.

\section*{Acknowledgements} I'd like to thank Dr. J. M. Landsberg for his unwavering support and guidance. Without his insight into this subject and patient oversight of my work, this paper would not have been possible.

\bibliography{citations}

\end{document}